\documentclass[11pt]{article}
\usepackage{latexsym,amssymb,mathrsfs}
\usepackage{amssymb}
\usepackage{fancybox}
\usepackage{mathrsfs}
\usepackage{cite}
\usepackage{amsmath,amsthm,enumerate,amsfonts,mathrsfs,ascmac,color}
\allowdisplaybreaks

\usepackage{amsthm}
\newtheorem{Theorem}{Theorem}[section]
\newtheorem{Corollary}{Corollary}[section]
\newtheorem{Proposition}{Proposition}[section]
\newtheorem{Lemma}{Lemma}[section]
\theoremstyle{definition}
\newtheorem{Remark}{Remark}[section]
\newtheorem{Definition}{Definition}[section]
\newtheorem{Problem}{Problem}

\numberwithin{equation}{section}

\newcommand{\N}{\mathbb{N}}
\newcommand{\R}{\mathbb{R}}
\newcommand{\pd}{\partial}
\newcommand{\lap}{\Delta}
\newcommand{\va}{\alpha}

\newcommand{\ve}{\varepsilon}

\begin{document}

\title{Positivity of solutions to the Cauchy problem for linear and semilinear
biharmonic heat equations}
\author{
\qquad\\
Hans-Christoph Grunau, Nobuhito Miyake and Shinya Okabe 
}
\date{}
\maketitle
\begin{abstract}
This paper is concerned with the positivity of solutions to the Cauchy problem for 
linear and nonlinear parabolic equations with the biharmonic operator as 
fourth order elliptic principal part. 
Generally, Cauchy problems for  parabolic equations of fourth order have no positivity preserving property due to the change of sign of the fundamental solution. 
One has \emph{eventual local} positivity for positive initial data, but on short time scales, one will in general have also regions of negativity.

The first goal of this paper is to find sufficient conditions on initial data which ensure the existence of solutions to the Cauchy problem for the linear biharmonic heat equation which are  \emph{positive for all times and in the whole space}. 

The second goal is to apply these results to show existence of \emph{globally} positive solutions to 
the Cauchy problem for a semilinear biharmonic parabolic equation. 
\end{abstract}
\vspace{25pt}
\noindent Addresses:

\smallskip
\noindent H.-Ch. G.: Fakult\"{a}t f\"{u}r Mathematik, Otto-von-Guericke-Universit\"{a}t, Postfach 4120, \\39016 Magdeburg, Germany.\\
\noindent 
E-mail: {\tt hans-christoph.grunau@ovgu.de}\\

\smallskip
\noindent N. M.:  Mathematical Institute, Tohoku University,
Aoba, Sendai 980-8578, Japan.\\
\noindent 
E-mail: {\tt nobuhito.miyake.t2@dc.tohoku.ac.jp}\\

\smallskip
\noindent 
\noindent S. O.:  Mathematical Institute, Tohoku University,
Aoba, Sendai 980-8578, Japan.\\
\noindent 
E-mail: {\tt shinya.okabe@tohoku.ac.jp}\\
\vspace{20pt}
\vspace{3pt}
\newpage

\section{Introduction} \label{Section:1}
This paper is concerned with the positivity of solutions to Cauchy problems for fourth order parabolic equations.

We say that a parabolic Cauchy problem has a \emph{positivity preserving property}
if non-negative and non-trivial initial data always yield  solutions 
which are  positive in the whole space and for any positive time.
It is well known that second order parabolic Cauchy problems enjoy a 
positivity preserving property.

On the other hand, it follows from \cite[Theorems 7.1 and 9.2]{CG} 
that the elliptic operator being of second order is not only sufficient 
but also \emph{necessary} for the corresponding Cauchy problem to 
enjoy a positivity preserving property. 
This means that this  property does not hold 
for the Cauchy problem for the biharmonic heat equation (see also \cite{Be, GG, FGG}):
\begin{alignat}{2}
\label{eq:1.1}
 & \pd_t u+(-\lap)^2 u =0 & &\quad \text{in} \quad \R^N\times(0, \infty),\\
\label{eq:1.2}
 & u(\cdot, 0) =\varphi(\cdot) & &\quad \text{in} \quad \R^N,
\end{alignat}
where $\varphi$ is a suitable measurable function and $N\ge1$.
``Suitable'' means locally integrable and less than exponential growth at infinity.
One should keep in mind that \emph{small times} are particularly sensitive for 
change of sign. For large times, at least in \emph{bounded domains}, the behaviour is more and 
more dominated by the elliptic principal part (and a strictly positive first eigenfunction would yield eventually positive solutions to the initial boundary value problem). 

The loss of the positivity preserving property for \eqref{eq:1.1}-\eqref{eq:1.2} is reflected by the sign change of the fundamental solution $G(\, \cdot\, , t)$ of the operator $\pd_t+(-\lap)^2$ in $\R^N\times(0, \infty)$
for all $t>0$. See Section~\ref{Subsection:2.1} below.
Moreover, it was even shown in \cite[Theorem~1]{GG} that for {\it any} non-negative and non-trivial function $\varphi\in C^\infty_{\rm c}(\R^N)$ there exists $T>0$ satisfying the following:
\begin{equation}
\label{eq:1.3}
 \inf_{x\in\R^N}[S(t)\varphi](x)<0
\end{equation}
for all $t\ge T$, where
\begin{equation}\label{eq:1.4}
 [S(t)\varphi](x):=\int_{\R^N}G(x-y, t)\varphi(y)\,dy
\end{equation}
solves  \eqref{eq:1.1}-\eqref{eq:1.2} 
for $(x, t)\in\R^N\times(0, \infty)$.

On the other hand, thinking of the biharmonic heat equation as a kind of linearised 
surface diffusion equation one would expect solutions to \eqref{eq:1.1}-\eqref{eq:1.2} 
for positive initial data to be \emph{on the whole positive}. 
Indeed, in \cite[Theorem~1]{GG}, it was proved that solutions to problem~\eqref{eq:1.1}-\eqref{eq:1.2} with non-negative non-trivial initial data $\varphi\in C^\infty_{\rm c}(\R^N)$ are eventually locally positive, that is, for any compact set $V\subset \R^N$ there exists $T=T(V)>0$ such that
\begin{equation*}
[S(t)\varphi](x)>0\quad \text{for}\quad  (x, t)\in V\times[T, \infty).
\end{equation*}
The issue of \emph{eventual local positivity} was studied further  in \cite{FGG} for
initial data with specific polynomial decay at inifinty: For $\beta>0$, initial data
\begin{equation}
\label{eq:1.5}
 \varphi(x):=\dfrac{1}{|x|^\beta+g(x)}
\end{equation}
with 
\begin{equation*}
 g\in A_\beta:=\{h\in C(\R^N)\mid h(x)>0, \ h(x)=o(|x|^{\beta})\ \text{as} \ |x|\to\infty\}
\end{equation*}
were considered. It was proved in \cite[Theorem~1.1]{FGG} 
that eventual local positivity holds locally uniformly 
and at an explicit asymptotic decay rate. At the same time this eventual positivity 
cannot be expected to be global, see \cite[Theorem~1.2]{FGG}: 
For each $\beta\in (0, N)$ and $t>1$ there exists a radially symmetric function $g\in A_\beta$  such that \eqref{eq:1.3} holds for $\varphi$ as in \eqref{eq:1.5}.

In order to understand the underlying reason for this change of sign even for large times
and how initial data could look like to avoid this, a first step was made also in \cite{FGG}:
\begin{Proposition}[{\cite[Proposition~A.6]{FGG}}]
\label{Proposition:1.1}
 Let $N=1$ and $\varphi(x):=|x|^{-\beta}$.
 For $\beta>0$ small enough, it holds that
 \begin{equation*}
  [S(t)\varphi](x)>0 \quad \text{for all}\quad  (x, t)\in\R\times(0, \infty).
 \end{equation*}
\end{Proposition}

So, it is natural to ask the following general question like Barbatis and Gazzola in \cite[Problem 13]{BG}:
\begin{Problem}
\label{Problem:A}
 For $N\ge1$, can one find suitable classes of initial data $\varphi$ 
 such that the corresponding solutions \eqref{eq:1.4}
 to \eqref{eq:1.1}-\eqref{eq:1.2}
  are globally positive{\rm ?}
\end{Problem}

To the best of our knowledge, the existence of \emph{globally (in space)} positive solutions to \eqref{eq:1.1}-\eqref{eq:1.2} has received only little attention. Beside Proposition~\ref{Proposition:1.1}, we mention Berchio's paper  \cite{B}. In \cite[Theorem~11]{B}  she considered
the initial datum 
$\varphi(x):=|x|^{-\beta}$ for $\beta\in (0,N)$
and introduced a right hand side with a strictly positive
impact. (The reader should notice that the actual formulation of \cite[Theorem~11]{B}
is not correct. A vanishing right hand side e.g. is not admissible.) 
In this situation she 
obtained \emph{eventual global} positivity.

In Theorem~\ref{Theorem:1.2} below we shall prove that this positivity is even 
global in time (i.e. even for arbitrarily small $t>0$) and holds for the 
\emph{homogeneous} biharmonic heat equation, provided that 
$\beta >0$ is small enough.

This, however, will follow from our first result which gives an affirmative answer to Problem~\ref{Problem:A}. 

Let $\mathcal{S}$ be the Schwartz space and $\mathcal{S}'$ be the space of tempered distributions.
We define $[S(t)\varphi](x)$ for $\varphi\in \mathcal{S}'$ as
\begin{equation*}
 [S(t)\varphi](x):=\langle\varphi, G(x-\cdot, t)\rangle
\end{equation*}
for $(x, t)\in\R^N\times(0, \infty)$, where $\langle\cdot, \cdot\rangle$ is the    duality pairing between $\mathcal{S}'$ and $\mathcal{S}$.
For $\varphi\in \mathcal{S}'$ we denote by $\mathcal{F}[\varphi]$ the Fourier transform of $\varphi$. If $\varphi\in \mathcal{S}'$ is even smooth then this is given by:
 \begin{equation}\label{eq:1.6}
  \mathcal{F}[\varphi](\xi):=(2\pi)^{-N/2}\int_{\R^N}e^{-ix\cdot \xi}\varphi(x)\,dx\quad\text{for} \quad \xi\in\R^N.
 \end{equation}
\begin{Theorem}
\label{Theorem:1.1}
 Let $N\ge3$ and $\varphi\in \mathcal{S}'$.
 Assume that all of the following conditions hold$\colon$
 \begin{itemize}
  \item[{\rm (a)}]
   $e^{-t|\cdot|^4}\mathcal{F}[\varphi]\in L^1(\R^N)$ for $t\in(0, \infty)$.
  \item[{\rm (b)}]
   $\mathcal{F}[\varphi]$ is real valued, radially symmetric and positive.
  \item[{\rm (c)}]
   $\psi(s):=s^{\frac{N-1}{2}}\mathcal{F}[\varphi](s)$ belongs to $C^1(0, \infty)$ and $\psi'(s)\le0$ for $s\in(0, \infty)$.
 \end{itemize}
 Then $[S(t)\varphi](x)$ is positive for $(x, t)\in\R^N\times(0, \infty)$.
\end{Theorem}
Theorem~\ref{Theorem:1.1} gives a general sufficient condition for the existence of positive solutions to problem~\eqref{eq:1.1}-\eqref{eq:1.2} when $N\ge3$. 
We remark that for sufficiently small $\beta>0$ the function $\varphi(x):=|x|^{-\beta}$ satisfies the assumptions on Theorem~\ref{Theorem:1.1} (for details, see Section~\ref{Subsection:3.2} and in particular \eqref{eq:3.15}).
Taking advantage of recurrence relations we can prove for this initial datum 
even in any dimension $N\ge 1$: 
\begin{Theorem}
\label{Theorem:1.2}
 Let $N\ge1$ and $\varphi(x):=|x|^{-\beta}$.
 \begin{itemize}
 \item[{\rm (i)}]
  There exist $\beta_1$, $\beta_2\in(0, N)$ with $\beta_1\le \beta_2$ 
  and
 $$
\beta_1 >
\left\{
\begin{array}{ll}
\displaystyle (N+1)/2 \quad & \quad \mbox{if}\quad N\ge 3,\\
\displaystyle 1/2 \quad & \quad \mbox{if}\quad N=2,\\
\displaystyle 7/16 \quad & \quad \mbox{if}\quad N=1,\\
\end{array}
\right.
$$ 
such that
  \begin{alignat}{2}
  \label{eq:1.7}
   [S(t)\varphi](x)>0 \quad &\text{in}\quad \R^N\times(0, \infty)  & & \quad \text{if}\quad \beta\in(0, \beta_1),\\
  \label{eq:1.8}
   \inf_{(x, t)\in \R^N\times(0, \infty)}&[S(t)\varphi](x)<0  & &\quad \text{if}\quad \beta\in(\beta_2, N). 
  \end{alignat}
 \item[{\rm (ii)}]
  Assume that $[S(t)\varphi](x)>0$ for $(x, t)\in\R^N\times(0, \infty)$.
  Then there exists $K_*=K_*(N, \beta)>0$ such that
  \begin{equation}
  \label{eq:1.9}
   [S(t)\varphi](x)\ge \dfrac{K_*}{|x|^\beta+t^{\beta/4}} \quad \text{for}\quad (x, t)\in\R^N\times(0, \infty).
  \end{equation}
   \item[{\rm (iii)}] For any $\beta\in (0,N)$ there  exists $K^*=K^*(N, \beta)>0$ such that
    \begin{equation}
  \label{eq:1.10}
   \Big| [S(t)\varphi](x)\Big|\le \dfrac{K^*}{|x|^\beta+t^{\beta/4}} \quad \text{for}\quad (x, t)\in\R^N\times(0, \infty).
  \end{equation}
 \end{itemize}
\end{Theorem}

In particular, Theorem~\ref{Theorem:1.2}~(i) gives an extension of Proposition~\ref{Proposition:1.1}.
Moreover, we deduce from \eqref{eq:1.8} that the condition $\beta\in(0, \beta_1)$ cannot be extended to $\beta\in(0, N)$.

Moreover, Theorem~\ref{Theorem:1.2} is  applied to show (to the best of our knowledge
for the first time) the existence of global-in-time \emph{positive} solutions to the Cauchy problem for the following fourth order semilinear parabolic equation:
\begin{alignat}{2}
\label{eq:1.11}
 & \pd_t u+(-\lap)^2 u =|u|^{p-1}u & &\quad \text{in} \quad \R^N\times(0, \infty),\\
\label{eq:1.12}
 & u(\cdot, 0) =\varepsilon\varphi(\cdot) & & \quad \text{in}\quad \R^N,
\end{alignat}
where $N\ge1$, $\varphi>0$ is a ``suitable'' measurable function, $\varepsilon>0$ is a parameter, and 
$$   
p>1+\frac{4}{N}.
$$
This ``super-Fujita'' condition is necessary in order to have global positive solutions
because Egorov and coauthors showed in\ \cite[Theorem~1.1]{EGKP}   finite time blow up of any positive solution
in the ``sub-Fujita'' case $1<p\le 1+4/N$.
See the ground breaking work \cite{F} of Fujita for second order analogues.

We first make clear that we understand  the notion of solution to 
problem~\eqref{eq:1.11}-\eqref{eq:1.12} in the strong sense:

\begin{Definition}
\label{Definition:1.1}
 Let $\varphi$ be locally integrable and bounded at infinity and $\varepsilon>0$.
 We say that $u\in C((0, \infty); BC(\R^N))$ is a global-in-time solution to problem~\eqref{eq:1.11}-\eqref{eq:1.12} if $u$ satisfies
 \begin{equation}
 \label{eq:1.13}
  u(x, t)=\varepsilon[S(t)\varphi](x)+\int^t_0[S(t-s)F_p(u(s))](x)\,ds
 \end{equation}
 for $(x, t)\in\R^N\times(0,\infty)$, where $F_p(\xi):=|\xi|^{p-1}\xi$.
\end{Definition}
Here, $BC(\R^N)$ denotes the space of bounded continuous functions.

Global existence of presumably sign changing solutions for similar problems was studied first by Caristi and Mitidieri in \cite{CM}.
As for eventual local positivity the following  was proved in \cite[Theorem~1.4]{FGG}:
For $\varphi$ given by \eqref{eq:1.5} with $\beta\in(4/(p-1), N)$ and $g\in A_\beta$
and $\varepsilon>0$ small enough, there exists a global-in-time solution~$u$ to 
problem~\eqref{eq:1.11}-\eqref{eq:1.12}, which is eventually locally positive.
However, to the best of our knowledge, there is no result for the existence of 
\emph{globally} positive solutions to problem~\eqref{eq:1.11}-\eqref{eq:1.12}.
Therefore, similarly to Problem~\ref{Problem:A}, it is also natural to ask the following question:
\begin{Problem}
\label{Problem:B}
 Are there initial data $\varphi$ such that there exists 
 a global-in-time positive solution to problem~\eqref{eq:1.11}-\eqref{eq:1.12}{\rm ?}
\end{Problem}

As an application of Theorem~\ref{Theorem:1.2}~(ii), we have:
\begin{Theorem}
\label{Theorem:1.3}
 Let $N\ge1$ and $p>1+4/N$.
 Set $\beta:=4/(p-1)$ and $\varphi(x):=|x|^{-\beta}$.
 Assume that
 \begin{equation}
 \label{eq:1.14}
  [S(t)\varphi](x)>0 \quad \text{for}\quad (x, t)\in\R^N\times(0, \infty).
 \end{equation}
 Then for sufficiently small $\varepsilon>0$, there exists a global-in-time solution $u$ to problem~\eqref{eq:1.11}-\eqref{eq:1.12} such that
 \begin{equation}
  \label{eq:1.15}
   u(x, t)\ge \dfrac{\varepsilon M_*}{|x|^\beta+t^{\beta/4}}\quad \text{for}\quad (x, t)\in\R^N\times(0, \infty),
 \end{equation}
 where $M_*>0$ depends only on $N$ and $p$.
\end{Theorem}
Theorem~\ref{Theorem:1.2}~(i) implies that, for each 
\begin{equation}
\label{eq:1.16}
 p>1+\dfrac{4}{\beta_1},
\end{equation}
condition \eqref{eq:1.14} holds true.
Thus Theorem~\ref{Theorem:1.3} gives an affirmative answer to Problem~\ref{Problem:B}, even though under the restriction \eqref{eq:1.16}.

 Let $\varphi$ be as in Theorem~{\rm \ref{Theorem:1.3}}.
 Then $\varphi$ belongs to the weak Lebesgue space $L^{r_c, \infty}(\R^N)$, where 
 \begin{equation*}
  r_c:=\dfrac{N(p-1)}{4}>1.
 \end{equation*}
 The existence of a global-in-time solution to problem~\eqref{eq:1.11}-\eqref{eq:1.12} with sufficiently small $\varepsilon>0$ and $\varphi\in L^{r_c, \infty}(\R^N)$ is obtained in \cite[Theorem~3.4 and Remark~3.7]{FV-R} (see also {\rm \cite[Theorem 1.1]{IKK}}).
However, in order to prove Theorem~\ref{Theorem:1.3}, we need to study the decay of global-in-time solution to \eqref{eq:1.11}-\eqref{eq:1.12} (which are not necessarily positive).

\begin{Theorem}
\label{Theorem:1.4}
 Let $N\ge1$ and $p>1+4/N$.
 Set $\beta:=4/(p-1)$ and $\varphi(x):=|x|^{-\beta}$.
 Then for sufficiently small $\varepsilon>0$, there exists a global-in-time solution $u$ to problem~\eqref{eq:1.11}-\eqref{eq:1.12} satisfying the following$\colon$
 There exists $M^*=M^*(N, p)>0$ such that
 \begin{equation}
 \label{eq:1.17}
  |u(x, t)|\le \dfrac{\varepsilon M^*}{|x|^\beta+t^{\beta/4}}\quad \text{for}\quad (x, t)\in\R^N\times(0, \infty).
 \end{equation}
\end{Theorem}

\begin{Remark}
	It is a natural question to ask whether our results  can be generalised to 
	Cauchy problems where the biharmonic operator is replaced by the polyharmonic operator 
	$(-\Delta)^m$ with $m > 1$. For related questions, results in this direction have already 
	been obtained. Indeed,  
	Ferreira and  Villamizar-Roa show in \cite{FV-R} well-posedness for
	problem~\eqref{eq:1.11}-\eqref{eq:1.12} with $(-\Delta)^m$  instead of $(-\Delta)^2$ and 
	assuming $p>1+(2m)/N$. They allow even for 
	any fractional $m >0$. Concerning problem \eqref{eq:1.1}-\eqref{eq:1.2} with $(-\Delta)^m$ 
	instead of $(-\Delta)^2$ and compactly supported nonnegative initial datum, 
	Ferreira and Ferreira prove in \cite{FF} eventual local positivity for any fractional polyharmonic operator (i.e. $m>1$) thereby solving \cite[Problem 10]{BG}
	mentioned by Barbatis and Gazzola. In view of the techniques developed in these papers we are 
	confident that the present paper can be extended to the general polyharmonic framework.
\end{Remark}

The rest of this paper is organised as follows. 
In Section~\ref{Section:2} we recall several properties of the fundamental solution $G$, of the Fourier transform of radially symmetric functions, and of Bessel functions.
In Section~\ref{Section:3} we prove Theorems~\ref{Theorem:1.1} and \ref{Theorem:1.2}. 
Section~\ref{Section:4} is devoted to the proofs of Theorems~\ref{Theorem:1.3} and \ref{Theorem:1.4}.

\section{Preliminaries} \label{Section:2}

In this section, we recall some properties of the fundamental solution $G$, of the Fourier transform of radially symmetric functions, and of Bessel functions 
which will be useful in order to prove our results.

\subsection{Fundamental solution $G$} \label{Subsection:2.1}

We collect properties of the fundamental solution $G$ without  proof (for  details, see e.g. \cite{FGG,GG,GP}).
Let $J_\mu$ be the   $\mu$-th Bessel function of the first kind.
Then $G$ is given by
\begin{equation*}
 G(x, t)=\dfrac{\va_N}{t^{N/4}}f_N\biggl(\dfrac{|x|}{t^{1/4}}\biggr)
\end{equation*}
for $x\in\R^N$ and $t>0$, where $\alpha_N:=(2\pi)^{-N/2}$ is a normalisation constant and
\begin{equation}
\label{eq:2.1}
 \begin{split}
  f_N(\eta):\!\!
  &=\eta^{1-N}\int^\infty_0e^{-s^4}(\eta s)^{N/2}J_{(N-2)/2}(\eta s)\,ds\\
  &=\eta^{-N}\int^\infty_0\exp\biggl[-\dfrac{s^4}{\eta^4}\biggr]s^{N/2}J_{(N-2)/2}(s)\,ds
 \end{split}
\end{equation}
for $\eta>0$.
It is known that $f_N$ changes sign infinitely many times, see \cite[Theorem 2.3]{FGG}.

In what follows the constants $c_i>0$ $(i=1, 2, 3)$  depend only on $N$.
\begin{itemize}
\item
 For $t>0$, the function $G (\cdot, t)$ belongs to Schwartz space $\mathcal{S}$.
 More precisely, $f_N$ satisfies
 \begin{equation}
 \label{eq:2.2}
  f_N'(\eta)=-\eta f_{N+2}(\eta),\quad|f_N(\eta)|\le c_1\exp\Bigl[-c_2\eta^{4/3}\Bigr], \quad \text{for}\quad \eta>0.
 \end{equation}
\item
 For $t>0$, it holds that
 \begin{equation}
 \label{eq:2.3}
  \mathcal{F}[G(\cdot, t)](\xi)=(2\pi)^{-N/2} e^{-|\xi|^4t}
 \end{equation}
 for $\xi\in\R^N$.
 Here, $\mathcal{F}$ denotes the Fourier transform defined in \eqref{eq:1.6}.
\end{itemize}

\subsection{Fourier transform of radially symmetric function} \label{Subsection:2.2}

To show positivity of $S(t)\varphi$, we use the representation of the Fourier transform of radially symmetric functions. 
According to \cite[Theorem 9.10.5]{AAR}
the Fourier transform of $f(x)=g(|x|)\in L^1(\R^N)$ is given by
\begin{equation}
\label{eq:2.4}
 \mathcal{F}[f](\xi)=|\xi|^{-(N-2)/2}\int^\infty_0s^{N/2}g(s)J_{(N-2)/2}(|\xi|s)\,ds
\end{equation}
for $\xi\in\R^N$.
Moreover, $\mathcal{F}[f]$ is also radially symmetric.
In what follows, we write $\mathcal{F}[f](\xi)=\mathcal{F}[f](|\xi|)$.

\subsection{Properties of Bessel functions} \label{Subsection:2.3}

We collect some properties of Bessel functions from  \cite[Chapter 4]{AAR}.
The Bessel function $J_\mu$ (of the first kind) satisfies the formulas
\begin{alignat}{2}
\label{eq:2.5}
 J_\mu(\eta)
 &=\sum_{k=0}^\infty\dfrac{(-1)^{k}}{\Gamma(k+1)\Gamma(k+\mu+1)}\Bigl(\dfrac{\eta}{2}\Big)^{2k+\mu}
 & &\quad \text{if}\quad \mu>-1,\\
\label{eq:2.6}
 J_{\mu}(\eta)
 &=\dfrac{1}{\sqrt{\pi}\Gamma(\mu+1/2)}\Bigl(\dfrac{\eta}{2}\Bigr)^\mu\int^\pi_0\cos(\eta\cos\theta)\sin^{2\mu}\theta\,d\theta
 & &\quad \text{if}\quad \mu>-1/2,
\end{alignat}
for $\eta>0$. See \cite[(4.5.2) and Corollary 4.11.2]{AAR}.
In particular, we observe from \eqref{eq:2.5} with $\mu=-1/2$ that
\begin{equation}
\label{eq:2.7}
 J_{-\frac{1}{2}}(\eta)=\sqrt{\dfrac{2}{\pi \eta}}\cos\eta
\end{equation}
for $\eta>0$.
It follows from \eqref{eq:2.5} that
\begin{equation}
\label{eq:2.8}
 J_\mu(\eta)=(\mu+1)\eta^{-1}J_{\mu+1}(\eta)+J'_{\mu+1}(\eta),
\end{equation}
\begin{equation}
\label{eq:2.9}
 J_\mu'(\eta)=\mu\,\eta^{-1}J_{\mu}(\eta)-J_{\mu+1}(\eta),
\end{equation}
\begin{equation}
\label{eq:2.10}
\lim_{\eta\searrow0}\eta^{-\mu}J_\mu(\eta)=2^{-\mu} \Gamma(\mu+1)^{-1},
\end{equation}
for $\mu>-1$.
Moreover, \eqref{eq:2.6} and \eqref{eq:2.7} imply that if $\mu\ge-1/2$, then
\begin{equation}
\label{eq:2.11}
 \sup_{0<\eta<\infty}\eta^{-\mu}|J_\mu(\eta)|<\infty.
\end{equation}
For large $\eta$, we also have the following asymptotic expansion for $\mu>-1$:
\begin{equation*}
 J_\mu(\eta)=\sqrt{\dfrac{2}{\pi\eta}}\cos\Bigl(\eta-\dfrac{\mu\pi}{2}-\dfrac{\pi}{4}\Bigr)+O(\eta^{-3/2})\quad \text{as}\quad \eta\to\infty.
\end{equation*}
Then we see that for $\mu\ge -1/2$
\begin{equation}
\label{eq:2.12}
 \sup_{0<\eta<\infty}\eta^{1/2}|J_\mu(\eta)|<\infty.
\end{equation}

We recall a monotonicity property of Bessel functions (\cite[Theorem 5.2]{LMS}).
Let $\{j_{\mu, k}\}^\infty_{k=1}$ be the zeroes of $J_\mu$ satisfying
\begin{equation*}
 0<j_{\mu, 1}<j_{\mu, 2}<\cdots<j_{\mu, k}<j_{\mu, k+1}<\cdots
\end{equation*}
and $j_{\mu, 0}:=0$.
Set
\begin{equation*}
 M_{\mu, k}:=\int^{j_{\mu, k+1}}_{j_{\mu, k}}W(s)s^{1/2}|J_\mu(s)|\,ds, \quad k\in\N\cup\{0\}.
\end{equation*}

\begin{Proposition}[{\cite[Theorem 5.2]{LMS}}]
\label{Proposition:2.1}
 Let $\mu\ge 1/2$.
 Let $W:(0, \infty)\to\R$ satisfy
 \begin{equation}
 \label{eq:2.13}
  \begin{aligned}
   & W(\eta)>0 \quad \text{and} \quad W'(\eta)\le 0 & & \quad \text{if}\quad \mu>\dfrac{1}{2},\\
   & W(\eta)>0 \quad \text{and} \quad W'(\eta)< 0 & & \quad \text{if}\quad \mu=\dfrac{1}{2},
  \end{aligned}
 \end{equation}
 for $\eta>0$ and
 \begin{equation}
 \label{eq:2.14}
  W(\eta)=O(\eta^\ve)\quad \text{as} \quad \eta \searrow 0,
 \end{equation}
 where $\ve>-3/2-\mu$.
 Then
 \begin{equation} \label{eq:2.15}
  M_{\mu, k}>M_{\mu, k+1} \quad \text{for}\quad k\in\N\cup\{0\}.
 \end{equation}
\end{Proposition}

\begin{Remark}
\label{Remark:2.1}
 The assumption \eqref{eq:2.14} is required to show that the integral $M_{\mu, 0}$ converges {\rm (}see {\rm \cite[Section~(ii)]{LMS}}{\rm )}.
 Thus \eqref{eq:2.14} is omitted if the integral $M_{\mu, 0}$ converges.
\end{Remark}

\begin{Remark}
\label{Remark:2.2}
 Assume that
 \begin{equation*}
  \int^\infty_0 W(s)s^{1/2}|J_\mu(s)|\,ds<\infty.
 \end{equation*}
 Then, it holds that
 \begin{equation*}
  \int^\infty_0 W(s)s^{1/2}J_\mu(s)\,ds=\sum^\infty_{k=0}(-1)^kM_{\mu, k}=\sum^\infty_{k=0}(M_{\mu, 2k}-M_{\mu, 2k+1}).
 \end{equation*}
 Hence \eqref{eq:2.15} leads to the positivity of the integral in the left hand side of the above equation.
\end{Remark}

\section{Existence of positive solutions to problem~\eqref{eq:1.1}-\eqref{eq:1.2}} \label{Section:3}

In this section, we prove the sufficient condition on $\varphi$ to ensure $[S(t)\varphi](x)>0$ for $(x, t)\in\R^N\times(0, \infty)$.
In what follows, the letter $C$  denotes generic positive constants and they may have different values even within the same line. 

\subsection{General initial data} \label{Subsection:3.1}

This section is devoted to the proof of Theorem~\ref{Theorem:1.1}.
\begin{proof}[{\bf Proof of Theorem~\ref{Theorem:1.1}}]
 Since by \eqref{eq:2.3}
 \begin{equation*}
  \mathcal{F}^{-1}[G(x-\cdot, t)](\xi)
  =e^{ix\cdot\xi}\mathcal{F}[G(\cdot, t)](\xi)
  =(2\pi)^{-N/2} e^{-t|\xi|^4+ix\cdot\xi}
 \end{equation*}
 for $x\in\R^N$ and $\xi\in\R^N$, we deduce from  \eqref{eq:2.4} and the assumption in Theorem~\ref{Theorem:1.1} that
 \begin{equation}
 \label{eq:3.1}
  \begin{split}
   [S(t)\varphi](x)
   &=\langle \varphi, G(x-\cdot, t)\rangle\\
   &=(2\pi)^{-N/2}\langle \mathcal{F}[\varphi], e^{-t|\cdot|^4+ix\cdot(\cdot)}\rangle\\
   &=(2\pi)^{-N/2}\int_{{\R^N}}\mathcal{F}[\varphi](|\xi|)e^{-t|\xi|^4+ix\cdot\xi}\,d\xi\\
   &=(2\pi)^{-N/2}\int_{{\R^N}}\left( \mathcal{F}[\varphi](|\xi|)e^{-t|\xi|^4}\right) e^{-ix\cdot\xi}\,d\xi\\
   &= |x|^{-(N-2)/2}\int^\infty_0s^{N/2}\mathcal{F}[\varphi](s)e^{-ts^4}J_{(N-2)/2}(|x|s)\,ds\\
   &= |x|^{-(N-2)/2}\int^\infty_0\psi(s)e^{-ts^4}s^{1/2}J_{(N-2)/2}(|x|s)\,ds\\
   &= |x|^{-(N+1)/2}\int^\infty_0\psi(|x|^{-1}s)\exp\bigg[-\dfrac{ts^4}{|x|^4}\bigg]s^{1/2}J_{(N-2)/2}(s)\,ds
  \end{split}
 \end{equation}
 for $(x, t)\in\R^N\times(0, \infty)$.
 By Proposition~\ref{Proposition:2.1} and Remark~\ref{Remark:2.2} it holds that $[S(t)\varphi](x)$ is positive for $(x, t)\in \R^N\times(0, \infty)$ if $N\ge3$.
\end{proof}
\subsection{Special initial data $\varphi(x)=|x|^{-\beta}$} \label{Subsection:3.2}

In this section, we prove Theorem~\ref{Theorem:1.2}.
To this end we consider another representation of $S(t)\varphi$.
Let $\beta\in(0, N)$.
Since $\mathcal{F}[\varphi](x)=c_{N, \beta}|x|^{\beta-N}$ in the sense of tempered distribution (see e.g. \cite[Proposition 4.64]{Mi}, $c_{N, \beta}=2^{N/2-\beta}
\Gamma((N-\beta)/2)/\Gamma(\beta/2)$), by an  argument similar to that in {\eqref{eq:3.1} in the proof of Theorem~\ref{Theorem:1.1} we have
\begin{equation*}
 \begin{split}
  [S(t)\varphi](x)
  & = c_{N, \beta}|x|^{-(N-2)/2}\int^\infty_0e^{-s^4t}s^{\beta-N/2}J_{(N-2)/2}(|x|s)\,ds\\
  & = c_{N, \beta}|x|^{-(N-2)/2}t^{-\beta/4+(N-2)/8}\int^\infty_0e^{-s^4}s^{\beta-N/2}J_{(N-2)/2}(|x|t^{-1/4}s)\,ds
 \end{split}
\end{equation*}
for $(x, t)\in\R^N\times(0, \infty)$.
Setting $E(s):=e^{-s^4}$ and
\begin{equation}
\label{eq:3.2}
 F_{N, \beta}(\eta):=\eta^{\beta-(N-2)/2}\int^\infty_0E(s)s^{\beta-N/2}J_{(N-2)/2}(\eta s)\,ds,
\end{equation}
we see that
\begin{equation}
\label{eq:3.3}
 [S(t)\varphi](x)=c_{N, \beta}|x|^{-\beta}F_{N, \beta}(|x|t^{-1/4})
\end{equation}
for $(x, t)\in\R^N\times(0, \infty)$.
Thus, in order to prove Theorem~\ref{Theorem:1.2}, it suffices to show that $F_{N, \beta}>0$.

The positivity statement will then be a direct consequence of Proposition~\ref{Proposition:2.1} and Remark~\ref{Remark:2.2} provided that $N\ge 3$. In order to cover also 
the small dimensions $N=1,2$, we need some preparations.
We remark that by a change of variables $F_{N, \beta}$ satisfies
\begin{equation}
\label{eq:3.4}
 F_{N, \beta}(\eta)=\int^\infty_0E(\eta^{-1}s)s^{\beta-N/2}J_{(N-2)/2}(s)\,ds
\end{equation}
for $\eta>0$.
We remark that $F_{N, \beta}$ can be also defined for $\beta\ge N$.
In the following, we consider $F_{N,\beta}$ with $N\ge1$ and $\beta>0$.

Since we observe from \eqref{eq:2.5} that
\begin{equation*}
 \lim_{s\searrow0}s^{\beta-N/2}J_{N/2}(\eta s)=0,
\end{equation*}
we see that by \eqref{eq:2.8} and \eqref{eq:3.2}
\begin{equation}
\label{eq:3.5}
 \begin{split}
  F_{N, \beta}(\eta)
  & =\eta^{\beta-N/2}\int^\infty_0E(s)s^{\beta-N/2-1}\biggl[\dfrac{N}{2}J_{N/2}(\eta s)+s\dfrac{d}{ds}\Bigl(J_{N/2}(\eta s)\Bigr)\biggr]\,ds\\
  & =(N-\beta)\eta^{\beta-N/2}\int^\infty_0E(s)s^{\beta-N/2-1}J_{N/2}(\eta s)\,ds\\
  & \quad +4\eta^{\beta-N/2}\int^\infty_0E(s)s^{\beta-N/2+3}J_{N/2}(\eta s)\,ds\\
  & =(N-\beta)F_{N+2, \beta}(\eta)+4\eta^{\beta-N/2}\int^\infty_0E(s)s^{\beta-N/2+3}J_{N/2}(\eta s)\,ds,
 \end{split}
\end{equation}
for $\eta>0$, $N\ge1$ and $\beta>0$.
In the following two lemmas we study the asymptotic behaviour of $F_{N,\beta}$
at $0$ and at $\infty$.
\begin{Lemma}
\label{Lemma:3.1}
 For $N\ge1$ and $\beta>0$
 \begin{equation}
 \label{eq:3.6}
  \lim_{\eta\to\infty}\eta^{\beta-N/2}\int^\infty_0E(s)s^{\beta-N/2+3}J_{N/2}(\eta s)\,ds=0.
 \end{equation}
\end{Lemma}

\begin{proof}
 We prove this lemma by means of an inductive argument.
 Let $N\ge1$ and $\beta>0$.
 We first claim that for $k\in\N\cup\{0\}$ there exists $\{a^{k}_{l}\}^{k}_{l=0}\subset\R$ such that for $\eta>0$
 \begin{equation}
 \label{eq:3.7}
  \begin{split}
   & \int^\infty_0E(s)s^{\beta-N/2+3}J_{N/2}(\eta s)\,ds\\
   &\quad =\eta^{-k}\sum^k_{l=0}a^{k}_l\int^\infty_0E^{(l)}(s)s^{\beta+l+3-N/2-k}J_{N/2+k}(\eta s)\,ds.
  \end{split}
 \end{equation}
 It is clear that \eqref{eq:3.7} holds for $k=0$.
 
 Assume that \eqref{eq:3.7} holds for some $k_*\in\N\cup\{0\}$.
 Similarly to \eqref{eq:3.5}, we have
   \begin{align*}
   &\int^\infty_0E(s)s^{\beta-N/2+3}J_{N/2}(\eta s)\,ds\\
   &=\eta^{-k_*-1}\sum^{k_*}_{l=0}a^{k_*}_l\biggl[\biggl(\dfrac{N}{2}+k_*+1\biggr)
   \int^\infty_0E^{(l)}(s)s^{\beta+l+2-N/2-k_*}J_{N/2+k_*+1}(\eta s)\,ds\\
   &\quad+\int^\infty_0E^{(l)}(s)s^{\beta+l+3-N/2-k_*}\dfrac{d}{ds}\Bigl(J_{N/2+k_*+1}
   (\eta s)\Bigr)\,ds\biggr]\\
   &=\eta^{-k_*-1}\sum^{k_*}_{l=0}a^{k_*}_l\biggl[(N+2k_*-2-\beta-l)\int^\infty_0
   E^{(l)}(s)s^{\beta+l+2-N/2-k_*}J_{N/2+k_*+1}(\eta s)\,ds\\
   &\quad-\int^\infty_0E^{(l+1)}(s)s^{\beta+l+3-N/2-k_*}J_{N/2+k_*+1}(\eta s)\,ds\biggr]\\
   &=\eta^{-k_*-1}\biggl[ (N+2k_*-2-\beta)a^{k_*}_0\int^\infty_0E(s)s^{\beta+3-N/2-(k_*+1)}J_{N/2+k_*+1}(\eta s)\,ds\\
   &\quad +\sum^{k_*}_{l=1}\bigl((N+2k_*-2-\beta-l)a^{k_*}_l-a^{k_*}_{l-1}\bigr)\int^\infty_0E^{(l)}(s)s^{\beta+l+3-N/2-(k_*+1)}J_{N/2+k_*+1}(\eta s)\,ds\\
   &\quad-a^{k_*}_{k_*}\int^\infty_0E^{(k_*+1)}(s)s^{\beta+3-N/2}J_{N/2+k_*+1}(\eta s)\,ds\biggr]
  \end{align*}
 for $\eta>0$.
 Thus \eqref{eq:3.7} holds for $k=k_*+1$.
 Therefore, \eqref{eq:3.7} holds for $k\in\N\cup\{0\}$.
 
 We now turn to prove \eqref{eq:3.6}.
 We first consider the case $\beta\in(0, (N+1)/2)$.
 It follows from \eqref{eq:2.11} and \eqref{eq:2.12} that
 $$
  \sup_{0<\eta<\infty}\eta^{-\gamma} |J_{N/2}(\eta)|<\infty, \quad \gamma\in\biggl[-\dfrac{1}{2}, \dfrac{N}{2}\biggr],
 $$
 and we have
 \begin{equation}
 \label{eq:3.8}
  \eta^{\beta-N/2}\biggl|\int^\infty_0E(s)s^{\beta-N/2+3}J_{N/2}(\eta s)\,ds\biggr|\le C\eta^{\beta-N/2+\gamma}\int^\infty_0E(s)s^{\beta-N/2+\gamma+3}\,ds
 \end{equation}
 for $\eta>0$ and $\gamma\in[-1/2, N/2]$.
 Since $\beta\in(0, (N+1)/2)$, we see that
 $$
  I:=\biggl[-\dfrac{1}{2}, \dfrac{N}{2}\biggr]\cap\biggl(\dfrac{N}{2}-\beta-4, \dfrac{N}{2}-\beta\biggr)\neq\emptyset.
 $$
 Fix $\tilde{\gamma}\in I$.
 Taking $\gamma=\tilde{\gamma}$ in \eqref{eq:3.8}, we observe that the right hand side of \eqref{eq:3.8} goes to $0$ as $\eta\to\infty$.
 Therefore, \eqref{eq:3.6} holds for $\beta\in(0, (N+1)/2)$.
 Next we consider the case $\beta\ge (N+1)/2$.
 Fix $\tilde{k}\in\N$ such that
 \begin{equation}
 \label{eq:3.9}
  \beta-\dfrac{N+1}{2}<\tilde{k}<\beta-\dfrac{N+1}{2}+4.
 \end{equation}
 Taking $k=\tilde{k}$ in \eqref{eq:3.7}, we observe from \eqref{eq:2.12} that
  \begin{equation}
  \label{eq:3.10}
  \begin{split}
   &\eta^{\beta-N/2}\biggl|\int^\infty_0E(s)s^{\beta-N/2+3}J_{N/2}(\eta s)\,ds\biggr|\\
   &\quad\le\eta^{\beta-N/2-\tilde{k}}\sum^{\tilde{k}}_{l=0}|a^{\tilde{k}}_l|\int^\infty_0|E^{(l)}(s)|s^{\beta+l+3-N/2-\tilde{k}}|J_{N/2+\tilde{k}}(\eta s)|\,ds\\
   &\quad\le C\eta^{\beta-(N+1)/2-\tilde{k}}\sum^{\tilde{k}}_{l=0}|a^{\tilde{k}}_l|\int^\infty_0 |E^{(l)}(s)|s^{\beta+l+5/2-N/2-\tilde{k}}\,ds
  \end{split}
 \end{equation}
 for $\eta>0$.
 By \eqref{eq:3.9} the right hand side of \eqref{eq:3.10} goes to $0$ as $\eta\to\infty$.
 Thus \eqref{eq:3.6} follows also for $\beta\ge (N+1)/2$.
\end{proof}

\begin{Lemma}
\label{Lemma:3.2}
 For $N\ge1$ and $\beta\in(0, N)$ there exist constants $A_{N, \beta}$, $\tilde{A}_{N, \beta}>0$ such that
 \begin{alignat}{1}
 \label{eq:3.11}
  \lim_{\eta\to\infty}F_{N, \beta}(\eta)&=A_{N, \beta}, \\
 \label{eq:3.12}
  \lim_{\eta\searrow0}\eta^{-\beta}F_{N, \beta}(\eta)&=\tilde{A}_{N, \beta}.
 \end{alignat}
\end{Lemma}

\begin{proof}
 We first show \eqref{eq:3.11}.
 We claim that
 \begin{equation}
 \label{eq:3.13}
  \lim_{\eta\to\infty}F_{N, \beta}(\eta) \ \text{exists and is positive for}\ N\ge3\ \text{and}\ \beta\in\biggl(0, \dfrac{N-1}{2}\biggr).
 \end{equation}
 Since by \eqref{eq:2.11} and \eqref{eq:2.12}
 \begin{equation*}
  |E(\eta^{-1}s)s^{\beta-N/2} J_{(N-2)/2}(s)|\le
  \left\{
  \begin{aligned}
   & Cs^{\beta-1}
   & &\quad \text{if}\quad 0<s\le 1,\\
   & Cs^{\beta-(N+1)/2}
   & &\quad \text{if}\quad s>1,
  \end{aligned}
  \right.
 \end{equation*}
 we can apply the Lebesgue dominated convergence theorem for the right hand side of \eqref{eq:3.4} and obtain
 \begin{equation*}
  \lim_{\eta\to\infty}F_{N, \beta}(\eta)=A_{N, \beta}:=\int^\infty_0s^{\beta-N/2}J_{(N-2)/2}(s)\,ds.
 \end{equation*}
 Recalling that $N\ge3$ and $\beta\in(0, (N-1)/2)$, by Proposition~\ref{Proposition:2.1} and Remark~\ref{Remark:2.2} we see that $A_{N, \beta}$ is positive.
 
 We prove the general case  inductively.
 We claim that for $k\in\N$
 \begin{equation}
 \label{eq:3.14}
  \lim_{\eta\to\infty}F_{N, \beta}(\eta)\ \text{exists and is positive for}\ N\ge1\ \text{and}\ \beta\in\biggl(0, \min\biggl\{N, \dfrac{N-1}{2}+k\biggr\}\biggr).
 \end{equation}
 By \eqref{eq:3.5}, \eqref{eq:3.13} and Lemma~\ref{Lemma:3.1} we see that \eqref{eq:3.14} holds for $k=1$.
 If \eqref{eq:3.14} holds for some $k_*\in\N$, then \eqref{eq:3.14} with $k=k_*+1$ follows from \eqref{eq:3.5}, \eqref{eq:3.14} with $k=k_*$ and Lemma~\ref{Lemma:3.1}.
 Hence \eqref{eq:3.11} holds for $N\ge1$ and $\beta\in(0, N)$.
 
 We prove \eqref{eq:3.12}.
 By \eqref{eq:2.11} we have
 \begin{equation*}
  \Bigl|\eta^{-(N-2)/2}E(s)s^{\beta-N/2}J_{(N-2)/2}(\eta s)\Bigr|\le CE(s)s^{\beta-1}
 \end{equation*}
 for $\eta>0$ and $s>0$.
 Then by \eqref{eq:2.10} the Lebesgue dominated convergence theorem is applicable for the product of $\eta^{-\beta}$ and the right hand side of \eqref{eq:3.2}, and we obtain
 \begin{equation*}
  \lim_{\eta\searrow0}\eta^{-\beta}F_{N, \beta}(\eta)=\tilde{A}_{N, \beta}:=
  \dfrac{1}{\Gamma(N/2)2^{(N-2)/2}}\int^\infty_0E(s)s^{\beta-1}\,ds>0.
 \end{equation*}
 Thus the proof of Lemma~\ref{Lemma:3.2} is complete.
\end{proof}

We now turn to the proof of Theorem~\ref{Theorem:1.2}.

\begin{proof}[{\bf Proof of Theorem~\ref{Theorem:1.2}}]
 We first prove assertion~(i).
 By Proposition~\ref{Proposition:2.1} and Remark~\ref{Remark:2.2} 
 (which require that $1/2\le \mu =(N-2)/2$)
 we have
 \begin{equation}
 \label{eq:3.15}
  F_{N, \beta}(\eta)>0\quad \text{for}\quad  \eta>0\quad \text{if} \quad N\ge3 \quad \text{and}\quad \beta\in(0, \beta_0],
 \end{equation}
 where $\beta_0:=(N+1)/2$ for $N\ge3$.
 
 In the case $N=2$, we deduce from \eqref{eq:2.9} that
 \begin{equation*}
  \begin{split}
   \dfrac{d}{d\eta}\Bigl[\eta^{-\beta}F_{2, \beta}(\eta)\Bigr]
   & =\dfrac{d}{d\eta}\biggl[\int^\infty_0E(s)s^{\beta-1}J_{0}(\eta s)\,ds\biggr]\\
   & =-\int^\infty_0E(s)s^{\beta}J_{1}(\eta s)\,ds=-\eta^{-\beta-1}F_{4, \beta+2}(\eta)
  \end{split}
 \end{equation*}
 for $\eta>0$ and $\beta>0$.
 Since we have already proved in \eqref{eq:3.15} that $F_{4, \beta+2}(\eta)$ is positive for $\eta>0$ if $0<\beta\le 1/2$, the map $\eta\mapsto\eta^{-\beta}F_{2, \beta}(\eta)$ is decreasing on $(0,\infty)$.
 Hence it follows from Lemma~\ref{Lemma:3.2} that $F_{2, \beta}$ is positive if $\beta\in(0, \beta_0]$, where $\beta_0:=1/2$ if $N=2$.
 
 We  turn to the case $N=1$.
 By \eqref{eq:2.7} and \eqref{eq:3.4} we have
 \begin{equation*}
  \begin{split}
   F_{1, \beta}(\eta)
   & =\sqrt{\dfrac{2}{\pi}}\int^\infty_0E(\eta^{-1}s)s^{\beta-1}\cos s\,ds\\
   & =\sqrt{\dfrac{2}{\pi}}\int^\infty_0\biggl(1-\beta+\dfrac{4s^4}{\eta^4}\biggr)E(\eta^{-1}s)s^{\beta-2}\sin s\,ds
  \end{split}
 \end{equation*}
 for $\eta>0$.
 By a direct calculation, we see that the map $s\mapsto (1-\beta+4s^4)E(s)s^{\beta-2}$ is non-increasing if $0<\beta\le 7/16$.
 Thus $F_{1, \beta}(\eta)$ is positive for $\eta>0$ if $\beta\in(0, \beta_0]$, where $\beta_0:=7/16$ if $N=1$.
 
 We prove that we can extend the positivity result to $\beta>\beta_0$.
 Assume that there exist $\{\gamma_m\}^\infty_{m=1}\subset(\beta_0, \infty)$, $\{\eta_m\}^\infty_{m=1}\subset(0, \infty)$ such that
 \begin{equation*}
  \gamma_m\to\beta_0\quad \text{as} \quad m\to\infty, \quad F_{N, \gamma_m}(\eta_m)\le 0 \quad \text{for}\quad m\in\N.
 \end{equation*}
 If $\{\eta_m\}^\infty_{m=1}$ is bounded then $\eta_m$ converges, after passing to a subsequence, to some $\eta_0\in[0, \infty)$.
 Otherwise, a subsequence of $\{\eta_m\}^\infty_{m=1}$ goes to infinity.
 In what follows it is important that a careful inspection of the proofs 
 of Lemmas~\ref{Lemma:3.1} and \ref{Lemma:3.2} shows that the arguments 
 are uniform with respect to $\beta$ in a neighbourhood of $\beta_0$.
 By an argument similar to that of the proof of Lemma~\ref{Lemma:3.2}, if $\eta_0\neq0$, 
 \begin{equation*}
  \lim_{m\to\infty}F_{N, \gamma_m}(\eta_m)=
  \left\{
  \begin{aligned}
   & F_{N, \beta_0}(\eta_0)
   & & \quad \text{if}\quad \{\eta_m\}^\infty_{m=1} \quad \text{is bounded},\\
   & A_{N, \beta_0}
   & & \quad \text{otherwise.}
  \end{aligned}
  \right.
 \end{equation*}
 This contradicts  the positivity of $F_{N, \beta_0}$ or $A_{N, \beta_0}$, respectively.
 In the case $\eta_0=0$, it follows with the same arguments as in 
 Lemma~\ref{Lemma:3.2} that
 \begin{equation*}
  0\ge \lim_{m\to\infty}\eta_m^{-\gamma_m}F_{N, \gamma_m}(\eta_m)=\tilde{A}_{N, \beta_0}>0,
 \end{equation*}
 again a contradiction.
 Therefore, we can find $\beta_1>\beta_0$ which satisfies \eqref{eq:1.7}.
 One may observe that this argument even proves that the set $\{ \beta\in(0,N):
 \mbox{\ (\ref{eq:1.7}) is satisfied } \}$ is open in $(0,N)$.
 
 Finally, we show the existence of $\beta_2$ which satisfies \eqref{eq:1.8}.
 Since
  \begin{equation*}
  F_{N, N}(\eta)=\int^\infty_0E(\eta^{-1}s)s^{N/2}J_{(N-2)/2}(s)\,ds=\eta^Nf_N(\eta),
 \end{equation*}
 where $f_N$ is as in \eqref{eq:2.1}, $F_{N, N}$ has a nontrivial negative part.
 Since $F_{N, \beta}(\eta)$ is continuous with respect to $\beta$, $F_{N, \beta}$ has also a nontrivial negative part if $\beta<N$ is sufficiently closed to $N$.
 Therefore, we obtain $\beta_2\ge \beta_1$ which satisfies \eqref{eq:1.8}.
 
  \medskip
 We prove the assertion~(ii).
 It follows from the assumption in (ii) that $F_{N, \beta}$ is positive on $(0, \infty)$.
 By \eqref{eq:3.11} in Lemma~\ref{Lemma:3.2} we find $\eta_*>0$ such that 
 $$
 F_{N,\beta}(\eta) \ge \dfrac{1}{2} A_{N,\beta} \quad \text{for} \quad \eta \ge \eta_*. 
 $$
 Since $F_{N,\beta}$ is continuous in $(0,\infty)$, we find $K_1>0$ such that 
 \begin{equation}
 \label{eq:3.16}
 F_{N,\beta}(\eta) \ge K_1 \quad \text{for} \quad \eta \ge 1. 
 \end{equation}
 Setting $\eta = |x| t^{-1/4}$, we deduce from \eqref{eq:3.3} and \eqref{eq:3.16} that 
 \begin{equation}
 \label{eq:3.17}
  |x|^\beta [S(t)\varphi](x) \ge c_{N, \beta}K_1 
  \quad \text{for} \quad (x, t) \in \R^N \times (0, \infty) \quad \text{with} \quad |x|\ge t^{1/4}.
 \end{equation}
 On the other hand, by \eqref{eq:3.12} in Lemma \ref{Lemma:3.2} we find $K_2>0$ such that  
 \begin{equation}
 \label{eq:3.18}
 \eta^{-\beta} F_{N,\beta}(\eta) \ge K_2 \quad \text{for} \quad \eta \le 1.  
 \end{equation}
 Setting $\eta=|x| t^{-1/4}$ in \eqref{eq:3.3} again, we have that 
 \begin{equation}
 \label{eq:3.19}
  [S(t)\varphi](x)=c_{N, \beta}t^{-\beta/4}\eta^{-\beta}F_{N, \beta}(\eta). 
 \end{equation}
 Combining \eqref{eq:3.18} with \eqref{eq:3.19}, we obtain  
 \begin{equation}
 \label{eq:3.20}
  t^{\beta/4}[S(t)\varphi](x) \ge c_{N, \beta}K_2 
  \quad \text{for} \quad (x, t) \in \R^N \times (0, \infty) \quad \text{with} \quad |x|\le t^{1/4}.
 \end{equation}
 Let $K_*:=c_{N, \beta}\min\{K_1, K_2\}>0$.
 We define $\mathcal{R} : \R^N \times (0, \infty) \to \R$ as 
 $$
 \mathcal{R}(x,t):= 
 \begin{cases}
 \vspace{0.1cm}
 \dfrac{K_*}{|x|^{\beta}} & \quad \text{if} \quad |x| t^{-1/4} \ge 1, \\
 \dfrac{K_*}{t^{\beta/4}} & \quad \text{if} \quad |x| t^{-1/4} \le 1.  
 \end{cases}
 $$
 It is clear that 
 \begin{equation}
 \label{eq:3.21}
 \mathcal{R}(x,t) \ge \dfrac{K_*}{|x|^\beta + t^{\beta/4}} 
 \quad \text{for} \quad (x,t) \in \R^N \times (0, \infty). 
 \end{equation}
 Thus, by \eqref{eq:3.17}, \eqref{eq:3.20} and \eqref{eq:3.21}, we obtain \eqref{eq:1.9}.
 This shows (ii).
 
 \medskip
 By Lemma~{\rm \ref{Lemma:3.2}}, \eqref{eq:3.4} and \eqref{eq:3.19} we also obtain \eqref{eq:1.10} in (iii).
 Here, $K^*$ in \eqref{eq:1.10} is a constant depending only on $N$ and $\beta$.
 We remark that the upper bound holds irrespective of whether $[S(t)\varphi](x)$ is positive or not. This proves (iii).
 The proof of Theorem~\ref{Theorem:1.2} is complete.
\end{proof}

As a direct consequence of Theorem~\ref{Theorem:1.2}-(i) we have:
\begin{Corollary}
\label{Corollary:3.1}
 Let $N\ge1$, $\beta\in(0, N)$ and $1< q<N/(N-\beta)$.
 For $f\in L^q(\R^N)$ with $f\ge0$ a.e. in $\R^N$, set
 $$
  \psi(x):=\int_{\R^N}\dfrac{f(y)}{|x-y|^\beta}\,dy
 $$
 If $\beta\in(0, \beta_1)$, where $\beta_1$ is as in Theorem~{\rm \ref{Theorem:1.2}}-{\rm (i)}, then $[S(t)\psi](x)$ is positive for $(x, t)\in\R^N\times(0, \infty)$ and satisfies
 \begin{equation}
 \label{eq:3.22}
  |[S(t)\psi](x)|\le Ct^{-\frac{N}{4}\big(\frac{1}{q}+\frac{\beta}{N}-1\big)}\|f\|_{L^q(\R^N)}\quad\text{for}\quad (x, t)\in\R^N\times(0, \infty).
 \end{equation}
\end{Corollary}
\begin{proof}
 By the Hardy-Littlewood-Sobolev inequality (see e.g., \cite[Theorem 4.3]{LL}) we see that \eqref{eq:3.22} holds.
 From Fubini's theorem we deduce that
 \begin{equation*}
  \begin{split}
  [S(t)\psi](x)
  &=\int_{\R^N}\int_{\R^N}G(x-y, t)|y-z|^{-\beta}f(z)\,dz\,dy\\
  &=\int_{\R^N}\biggl(\int_{\R^N}G(x-y, t)|y-z|^{-\beta}\,dy\biggr)f(z)\,dz\\
  &=\int_{\R^N}[S(t)\varphi](x-z)f(z)\,dz,
 \end{split}
 \end{equation*}
 where $\varphi(x):=|x|^{-\beta}$.
 Then this together with Theorem~\ref{Theorem:1.2}-(i) implies that
 \begin{equation*}
  [S(t)\psi](x)>0\quad \text{for}\quad (x, t)\in\R^N\times(0, \infty).
 \end{equation*}
 Thus Corollary~\ref{Corollary:3.1} follows.
\end{proof}

\section{Global-in-time positive solutions to problem~\eqref{eq:1.11}-\eqref{eq:1.12}} \label{Section:4}

In this section, we consider the semilinear equation \eqref{eq:1.11} and prove Theorem~\ref{Theorem:1.3}.
Set
\begin{equation*}
 H(x, t):=\int^t_0\int_{\R^N}\exp\biggl[-c_2\Bigl(\dfrac{|y|}{s^{1/4}}\Bigr)^{4/3}\biggr]\dfrac{s^{-N/4}}{(|x-y|^\beta+(t-s)^{\beta/4})^p}\,dy\,ds
\end{equation*}
for $(x, t)\in\R^N\times(0, \infty)$, where $c_2$ is given by \eqref{eq:2.2}.
We remark that the function $H$ appears when we estimate the second term of the right hand side of \eqref{eq:1.13} by \eqref{eq:2.2} and
\begin{equation*}
 |u(x, t)|\le \dfrac{C}{|x|^\beta+t^{\beta/4}}\quad \text{for}\quad \R^N\times(0, \infty).
\end{equation*}
We first consider the decay estimate for $H$.

\begin{Proposition}
\label{Proposition:4.1}
 Let $N$, $p$ and $\beta$ be as in Theorem~{\rm \ref{Theorem:1.3}}.
 Then 
 \begin{equation}
 \label{eq:4.1}
  \sup_{(x, t)\in\R^N\times(0, \infty)}(|x|^\beta+t^{\beta/4})H(x, t)<\infty.
 \end{equation}
\end{Proposition}

\begin{proof}
 The proof is based on the argument in \cite[Proposition 6.2]{FGG}.
 We first claim that
 \begin{equation}
 \label{eq:4.2}
  \sup_{(x, t)\in\R^N\times(0, \infty)}t^{\beta/4}H(x, t)<\infty.
 \end{equation}
 By the change of variables $z=s^{-1/4}y$ and $\sigma=s^{-1}t$ we have
 \begin{equation*}
  \begin{split}
   t^{\beta/4}H(x, t)
   & = t^{\beta/4}\int^t_0\int_{\R^N}e^{-c_2|z|^{4/3}}\dfrac{s^{-\beta p/4}}{(|s^{-1/4}x-z|^\beta+(s^{-1}t-1)^{\beta/4})^p}\,dz\,ds\\
   & = t^{\beta/4+1-\beta p/4}\int^\infty_1\int_{\R^N}e^{-c_2|z|^{4/3}}\dfrac{\sigma^{\beta p/4-2}}{(|z-\sigma^{1/4}t^{-1/4}x|^\beta+(\sigma-1)^{\beta/4})^p}\,dz\,d\sigma
  \end{split}
 \end{equation*}
 for $(x, t)\in\R^N\times(0, \infty)$.
 Since $\beta=4/(p-1)$, we have $\beta/4+1-\beta p/4=0$ and
 \begin{equation}
 \label{eq:4.3}
  t^{\beta/4}H(x, t)=\int^\infty_1\int_{\R^N}H_{N, \beta}(z, \sigma; t^{-1/4}x)\,dz\,d\sigma.
 \end{equation}
 Here, we set
 \begin{equation*}
  H_{N, \beta}(z, \sigma; w):=e^{-c_2|z|^{4/3}}\dfrac{\sigma^{\beta p/4-2}}{(|z-\sigma^{1/4}w|^\beta+(\sigma-1)^{\beta/4})^p}.
 \end{equation*}
 We estimate the right hand side of \eqref{eq:4.3} by splitting the integral into three parts
 \begin{alignat*}{1}
  &A_1(w):=\int^\infty_2\int_{\R^N}H_{N, \beta}(z, \sigma; w)\,dz\,d\sigma,\\
  &A_2(w):=\int^2_1\int_{|z-\sigma^{1/4}w|\ge1/2}H_{N, \beta}(z, \sigma; w)\,dz\,d\sigma,\\
  &A_3(w):=\int^2_1\int_{|z-\sigma^{1/4}w|\le1/2}H_{N, \beta}(z, \sigma; w)\,dz\,d\sigma.
 \end{alignat*}
 Regarding $A_1$ and $A_2$, we have
 \begin{alignat}{1}
 \label{eq:4.4}
  A_1(w)
  &\le \int_{\R^N}e^{-c_2|z|^{4/3}}\,dz\int^\infty_2\sigma^{-2}\biggl(\dfrac{\sigma}{\sigma-1}\biggr)^{\beta p/4}\,d\sigma<\infty,\\
 \label{eq:4.5}
  A_2(w)
  &\le \int_{\R^N}e^{-c_2|z|^{4/3}}\,dz\int^2_1\dfrac{\sigma^{\beta p/4-2}}{(2^{-\beta}+(\sigma-1)^{\beta/4})^p}\,d\sigma<\infty.
 \end{alignat}
 We consider $A_3$.
 Recalling that $1-\beta p/4=-\beta/4<0$ and $0<\beta<N$, we see that
 \begin{equation}
 \label{eq:4.6}
  \begin{split}
   A_3(w)
    &\le C\int^2_1\int_{|\xi|\le1/2}\dfrac{1}{(|\xi|^\beta+(\sigma-1)^{\beta/4})^p}\,d\xi\,d\sigma\\
    &\le C\int_{|\xi|\le 1/2}\int^2_1\dfrac{1}{(\sigma-1+|\xi|^4)^{\beta p/4}}\,d\sigma\,d\xi\le C\int_{|\xi|\le 1/2}|\xi|^{-\beta}\,d\xi<\infty.
  \end{split}
 \end{equation}
 Combining \eqref{eq:4.4}, \eqref{eq:4.5} and \eqref{eq:4.6} with \eqref{eq:4.3}, we obtain \eqref{eq:4.2}.
 
 We prove  now that 
 \begin{equation}
 \label{eq:4.7}
  \sup_{(x, t)\in\R^N\times(0, \infty)}|x|^\beta H(x, t)<\infty.
 \end{equation}
The proof of \eqref{eq:4.7} is based on \cite[Proposition 6.2]{FGG} and \cite[Lemma 10]{GG07}. 
Using the change of variables $z=s^{-1/4}y$ and $\sigma=(|z|/|x|) s^{1/4}$, 
we deduce from Fubini's theorem that
\begin{equation*}
 \begin{split}
  |x|^{\beta}H(x, t)
  &=|x|^\beta\int^t_0\int_{\R^N}\exp[-c_2|z|^{4/3}]\dfrac{1}{(|x-s^{1/4}z|^\beta+(t-s)^{\beta/4})^p}\,dz\,ds\\
  &=|x|^\beta\int_{\R^N}\exp[-c_2|z|^{4/3}]\int^t_0\dfrac{1}{(|x-s^{1/4}z|^\beta+(t-s)^{\beta/4})^p}\,ds\,dz\\
  &=4\int_{\R^N}\dfrac{\exp[-c_2|z|^{4/3}]}{|z|^4}\int^{\frac{|z|}{|x|}t^{1/4}}_0\dfrac{\sigma^3}{\Big(\Big|\frac{x}{|x|}-\sigma \frac{z}{|z|}\Big|^\beta+\Big(\frac{t}{|x|^4}-\frac{\sigma^4}{|z|^4}\Big)^{\beta/4}\Big)^p}\,d\sigma\,dz.
 \end{split}
\end{equation*}
Without loss of generality, we may take $x/|x|={\bf e}_1:=(1, 0, \dots, 0)$.
Since
\begin{equation*}
 \bigg|{\bf e}_1-\sigma\dfrac{z}{|z|}\bigg|^\beta=\bigg(1-2\dfrac{z_1}{|z|}\sigma+\sigma^2\bigg)^{\beta/2}=\bigg((\sigma-1)^2+2\sigma\bigg(1-\dfrac{z_1}{|z|}\bigg)\bigg)^{\beta/2},
\end{equation*}
we have 
\begin{equation}
\label{eq:4.8}
|x|^{\beta}H(x, t) \le C \tilde{H}(R), 
\end{equation}
where $R:=t^{1/4}/|x|$ and 
\begin{equation*}
  \tilde{H}(R):=
   \int_{\R^N}\dfrac{\exp[-c_2|z|^{4/3}]}{|z|^4}\int^{R|z|}_0\dfrac{\sigma^3}{\Big(R^4-\frac{\sigma^4}{|z|^4}\Big)^{\beta p/4}+(\sigma-1)^{\beta p}+\Big(2\sigma\Big(1-\frac{z_1}{|z|}\Big)\Big)^{\beta p/2}}\,d\sigma\,dz. 
\end{equation*}
We first consider the case $N\ge2$.
Putting $z=(z_1, z')$ and changing the variables $r=|z'|$, we reduce $\tilde{H}(R)$ into 
\begin{equation*}
 \begin{split}
&\tilde{H}(R) = (N-1)\omega_{N-1}\int^\infty_0r^{N-2}\int^\infty_{-\infty}\dfrac{\exp[-c_2(|z_1|^2+r^2)^{2/3}]}{(|z_1|^2+r^2)^2}\\
& \quad \times\int^{R\sqrt{|z_1|^2+r^2}}_0\dfrac{\sigma^3}{\Big(R^4-\frac{\sigma^4}{(|z_1|^2+r^2)^2}\Big)^{\beta p/4}+(\sigma-1)^{\beta p}+\Big(2\sigma\Big(1-\frac{z_1}{\sqrt{|z_1|^2+r^2}}\Big)\Big)^{\beta p/2}}\,d\sigma\,dz_1\,dr.
 \end{split}
\end{equation*}
Here $\omega_{N-1}$ denotes the $(N-1)$-dimensional volume of $B_1(0)\subset \mathbb{R}^{N-1}$. 
By changing the variable $z_1=rw$, we observe that 
\begin{equation*}
 \begin{split}
&\tilde{H}(R) =(N-1)\omega_{N-1}\int^\infty_0r^{N-5}\int^\infty_{-\infty}\dfrac{\exp[-c_2r^{4/3}(1+w^2)^{2/3}]}{(1+w^2)^2}\\
  &\quad\times\int^{Rr\sqrt{1+w^2}}_0\dfrac{\sigma^3}{\Big(R^4-\frac{\sigma^4}{r^4(1+w^2)^2}\Big)^{\beta p/4}+(\sigma-1)^{\beta p}+\Big(2\sigma\Big(1-\frac{w}{\sqrt{1+w^2}}\Big)\Big)^{\beta p/2}}\,d\sigma\,dw\,dr.
 \end{split}
\end{equation*}
Since 
\begin{equation*}
 \begin{split}
  1-\dfrac{w}{\sqrt{1+w^2}}
   &=\dfrac{\sqrt{1+w^2}-w}{\sqrt{1+w^2}} 
   =\dfrac{1}{(\sqrt{1+w^2}+w)\sqrt{1+w^2}} 
   \ge\dfrac{1}{2(1+w^2)},\\
  R^4-\frac{\sigma^4}{r^4(1+w^2)^2}
  &=\dfrac{(Rr\sqrt{1+w^2}-\sigma)(Rr\sqrt{1+w^2}+\sigma)(R^2r^2(1+w^2)+\sigma^2)}{r^4(1+w^2)^2}\\
  &\ge\dfrac{R^3(Rr\sqrt{1+w^2}-\sigma)}{r\sqrt{1+w^2}},
 \end{split}
\end{equation*}
for $w\in\R$ and $\sigma\in[0, Rr\sqrt{1+w^2}]$, we have 
\begin{equation*}
 \begin{split}
\tilde{H}(R) & \le 2(N-1)\omega_{N-1}\int^\infty_0\int^\infty_{0}r^{N-5} \dfrac{\exp[-c_2r^{4/3}(1+w^2)^{2/3}]}{(1+w^2)^2}\\
  &\quad\times\int^{Rr\sqrt{1+w^2}}_0\dfrac{\sigma^3}{\Big(\frac{R^3(Rr\sqrt{1+w^2}-\sigma)}{r\sqrt{1+w^2}}\Big)^{\beta p/4}+(\sigma-1)^{\beta p}+\Big(\frac{\sigma}{1+w^2}\Big)^{\beta p/2}}\,d\sigma\,dr\,dw.
 \end{split}
\end{equation*}
Changing the variable $\rho=\rho(r)=r\sqrt{1+w^2}$ , we deduce from \eqref{eq:4.8} that for $N\ge 2$:
\begin{equation}
\label{eq:4.9}
 \begin{split}
  |x|^{\beta}H(x, t)
  &\le C\int^\infty_{0}\dfrac{1}{(1+w^2)^{N/2}}\int^\infty_0\rho^{N-5}e^{-c_2\rho^{4/3}}\\
  &\quad \times\int^{R\rho}_0\dfrac{\sigma^3}{\Big(\frac{R^3(R\rho-\sigma)}{\rho}\Big)^{\beta p/4}+(\sigma-1)^{\beta p}+\Big(\frac{\sigma}{1+w^2}\Big)^{\beta p/2}}\,d\sigma\,d\rho\,dw\\
  &\le C\int^\infty_{0}\dfrac{1}{(1+w^2)^{N/2}}\int^\infty_0\rho^{N-5}e^{-c_2\rho^{4/3}}\\
  &\quad \times\int^{R\rho}_0\dfrac{\sigma^3}{\Big(\frac{R^3(R\rho-\sigma)}{\rho}+(\sigma-1)^{4}+\frac{\sigma^2}{(1+w^2)^2}\Big)^{\beta p/4}}\,d\sigma\,d\rho\,dw.
 \end{split}
\end{equation}
Next we consider the case where $N=1$.
Since
\begin{equation*}
 R^4-\dfrac{\sigma^4}{|z|^4}=\dfrac{(R|z|+\sigma)(R|z|-\sigma)(R^2|z|^2+\sigma^2)}{|z|^4}\ge\dfrac{R^3(R|z|-\sigma)}{|z|},
\end{equation*}
for $\sigma\in[0, R\,|z|\, ]$, we have
\begin{equation*}
 \begin{split}
\tilde{H}(R)& \le 2\int^{\infty}_{0}z^{-4}e^{-c_2z^{4/3}}\int^{Rz}_0\dfrac{\sigma^3}{\Big(\frac{R^3(Rz-\sigma)}{z}\Big)^{\beta p/4}+(\sigma-1)^{\beta p}}\,d\sigma\,dz\\
    & \le C\int^{\infty}_{0}z^{-4}e^{-c_2z^{4/3}}\int^{Rz}_0\dfrac{\sigma^3}{\Big(\frac{R^3(Rz-\sigma)}{z}+(\sigma-1)^{4}\Big)^{\beta p/4}}\,d\sigma\,dz.
 \end{split}
\end{equation*}
This together with \eqref{eq:4.8} implies that in the case $N=1$:
\begin{equation}
\label{eq:4.10}
 |x|^{\beta}H(x, t)\le C\int^{\infty}_{0}z^{-4}e^{-c_2z^{4/3}}\int^{Rz}_0\dfrac{\sigma^3}{\Big(\frac{R^3(Rz-\sigma)}{z}+(\sigma-1)^{4}\Big)^{\beta p/4}}\,d\sigma\,dz.
\end{equation}
We recall that $R=t^{1/4}/|x|$.
Then it follows from \cite[Lemmas~7.1, 7.2]{FGG} with $\beta=4/(p-1)$ that the right hand sides of \eqref{eq:4.9} and \eqref{eq:4.10} are bounded. 
Thus \eqref{eq:4.7} follows. 
Combining \eqref{eq:4.2} with \eqref{eq:4.7}, we obtain \eqref{eq:4.1}. 
This completes the proof. 
\end{proof}

We now prove Theorems~\ref{Theorem:1.3} and \ref{Theorem:1.4}.

\begin{proof}[{\bf Proof of Theorem~\ref{Theorem:1.4}}]
 Let $\varepsilon>0$.
 We define a closed subset $(X, \| \, . \,\|)$ 
 of the corresponding  Banach space as follows:
 \begin{equation*}
  X:=\biggl\{v\in C(\R^N\times(0, \infty))\,\bigg|\,\| v\| \le 2\varepsilon K^*\biggl\},
 \end{equation*}
 \begin{equation*}
  \| v\| :=\sup_{(x, t)\in\R^N\times(0, \infty)}(|x|^\beta+t^{\beta/4})|v(x, t)|.
 \end{equation*}
 Here, $K^*$ is given in \eqref{eq:1.10}.
 Set
 \begin{equation*}
  \Phi[v](x, t):=\varepsilon[S(t)\varphi](x)+\int^t_0[S(t-s)F_p(v(\cdot, s))](x)\,ds\quad \text{for} \quad v\in X.
 \end{equation*}
 We find a fixed point of $\Phi$ on $X$ by the contraction mapping theorem.
 By \eqref{eq:1.10}, \eqref{eq:2.2} and Proposition~\ref{Proposition:4.1} we have
 \begin{equation*}
  \begin{split}
   (|x|^\beta+t^{\beta/4})|\Phi[v](x, t)|
   &\le K^*\varepsilon+(|x|^\beta+t^{\beta/4})\int^t_0|S(t-s)F_p(v(s))|\,ds\\
   &\le K^*\varepsilon+C\varepsilon^{p}(|x|^\beta+t^{\beta/4})H(x, t)\\
   &\le K^*\varepsilon(1+C\varepsilon^{p-1})
  \end{split}
 \end{equation*}
 for $v\in X$ and $(x, t)\in\R^N\times(0, \infty)$.
 Choosing $\varepsilon>0$ sufficiently small, we see that
 \begin{equation}
 \label{eq:4.11}
 \|\Phi[v]\| =
  \sup_{(x, t)\in \R^N\times(0, \infty)}(|x|^\beta+t^{\beta/4})|\Phi[v](x, t)|\le 2\varepsilon K^*
 \end{equation}
 for $v\in X$.
 By an argument similar  to that in \eqref{eq:4.11}, choosing $\varepsilon$ sufficiently small we have
 \begin{equation}
 \label{eq:4.12}
  \left\| \Phi[v]- \Phi[w]\right\|\le \dfrac{1}{2} \| v - w\|
 \end{equation}
 for $v$, $w\in X$.
 Thanks to \eqref{eq:4.11} and \eqref{eq:4.12} we obtain a unique fixed point $u\in X$ of $\Phi$ by the contraction mapping theorem.
Since $u\in X$, we obtain \eqref{eq:1.17}.
\end{proof}

\begin{proof}[{\bf Proof of Theorem~\ref{Theorem:1.3}}]
 Assume that $[S(t)\varphi](x)>0$ for $(x, t)\in\R^N\times(0, \infty)$.
 It follows from the same argument as in \eqref{eq:4.11} that $u$ satisfies
 \begin{equation*}
  \bigg|\int^t_0S(t-s)F_p(u(\cdot, s))](x)\,ds\bigg|\le\dfrac{C\varepsilon^{p}}{|x|^\beta+t^{\beta/4}}
 \end{equation*}
 for $(x, t)\in\R^N\times(0, \infty)$.
 This together with \eqref{eq:1.9} implies that
 \begin{equation*}
  u(x, t)\ge \dfrac{\varepsilon (K _*-C\varepsilon^{p-1})}{|x|^\beta+t^{\beta/4}}
 \end{equation*}
 for $(x, t)\in\R^N\times(0, \infty)$.
 Taking $\varepsilon>0$ small enough, we obtain \eqref{eq:1.15}.
 Therefore, the proof of Theorem~\ref{Theorem:1.3} is complete.
\end{proof}

\noindent
{\bf Acknowledgments.}
This work was initiated during the first author's visit at Tohoku University.
The first author is very grateful to second and third authors for their warm hospitality and the inspiring working atmosphere.
The second author was supported in part by the Grant-in-Aid for JSPS Fellows (No. JP19J10424) from Japan Society for the Promotion of Science.
The third author was supported in part by the Grant-in-Aid for Scientific Research(S) (No. JP19H05599) from Japan Society for the Promotion of Science.


\end{document}